\documentclass[a4paper, reqno]{amsart}
\usepackage{amssymb}
\usepackage{amsmath}
\usepackage{amsthm}
\usepackage[utf8]{inputenc}
\usepackage[T1]{fontenc}
\usepackage{lmodern}
\usepackage{ae}
\usepackage{todonotes}
\usepackage{graphicx}
\usepackage[displaymath,modulo,mathlines]{lineno}
\usepackage{enumerate}
\usepackage{fancybox}
\usepackage{setspace}
\usepackage[all,xdvi]{xy}
\usepackage{ifthen}
\usepackage{tikz, ulem, braids}

\usetikzlibrary{shapes,matrix,arrows}
\tikzstyle cross=[preaction={draw=white, -, line width=6pt}, thick]
\tikzstyle normal=[thick]
\tikzstyle zero=[very thick, gray]
\tikzstyle zerocross=[preaction={draw=white, -, line width=6pt}, very thick, gray]
\tikzstyle point=[draw,circle,inner sep=1,fill=black]
\tikzstyle petitpoint=[draw,circle,inner sep=0.3,fill=black]

 \usepackage[pdftex,
     bookmarks         = true,
     bookmarksnumbered = true,
     pdfpagemode       = None,
     pdfstartview      = FitH,
     pdfpagelayout     = SinglePage,
     urlcolor          = blue,
     pdfborder         = {0 0 0},
 	pdftoolbar = true
     ]{hyperref}                

\modulolinenumbers[2]
\numberwithin{equation}{section}

\newcommand{\C}{\mathbb{C}}

\newcommand{\Z}{\mathbb{Z}}

\newcommand{\kk}{\boldsymbol{ \operatorname{k}}}

\newcommand{\mf}[1]{\mathfrak{#1}}

\newcommand{\id}{\operatorname{id}}

\newcommand{\R}{\mathcal{R}}

\newcommand{\lef}{\vartriangleright}

\newcommand{\h}{\hslash}

\newcommand{\ot}{\otimes}

\newcommand{\modu}{\operatorname{-mod}}

\newcommand{\slz}{\widetilde{SL_2(\Z)}}
\newtheorem{thm}{Theorem}[section]
\newtheorem{cor}[thm]{Corollary}
\newtheorem{prop}[thm]{Proposition}

\newtheorem{defi}[thm]{Definition}

\theoremstyle{remark} 
\newtheorem{rmk}[thm]{Remark}

\address[a.brochier@ed.ac.uk]{Adrien Brochier, School of Mathematics and Maxwell Institute for Mathematical Sciences, University of Edinburgh, James Clerk Maxwell Building, Kings Buildings, Mayfield Road, Edinburgh EH9 3JZ, UK}
\address[d.jordan@ed.ac.uk]{David Jordan, School of Mathematics and Maxwell Institute for Mathematical Sciences, University of Edinburgh, James Clerk Maxwell Building, Kings Buildings, Mayfield Road, Edinburgh EH9 3JZ, UK}
\title[Fourier transform for quantum $D$-modules]{Fourier transform for quantum $D$-modules via the punctured torus mapping class group}
\author[A. Brochier, D. Jordan]{Adrien Brochier, David Jordan}

\begin{document}
\maketitle
\begin{abstract}
We construct a certain cross product of two copies of the braided dual $\tilde H$ of a quasitriangular Hopf algebra $H$, which we call the elliptic double $E_H$, and which we use to construct representations of the punctured elliptic braid group extending the well-known representations of the planar braid group attached to $H$.  We show that the elliptic double is the universal source of such representations. We recover the representations of the punctured torus braid group obtained in \cite{Jordan2009}, and hence construct a homomorphism from $E_H$ to the Heisenberg double $D_H$, which is an isomorphism if $H$ is factorizable.

The universal property of $E_H$ endows it with an action by algebra automorphisms of the mapping class group $\slz$ of the punctured torus.  One such automorphism we call the quantum Fourier transform; we show that when $H=U_q(\mathfrak{g})$, the quantum Fourier transform degenerates to the classical Fourier transform on $D(\mathfrak{g})$ as $q\to 1$.
\end{abstract}
\section{Introduction}
Let $(H,\R)$ be a quasi-triangular Hopf algebra, and let $\tilde H$ denote the braided dual -- also known as the reflection equation algebra -- of $H$~\cite{Donin2003a,Donin2002,Donin2003,Majid1995}.    This is the restricted dual vector space $H^\circ$, but the multiplication is twisted from the standard one by the $R$-matrix (see Section~\ref{sec:dual} for details).

Let $\{e_i\}$ and $\{e^i\}$ denote dual bases of $H$ and $\tilde H$, respectively. Then the canonical element $X=\sum e^i \ot e_i\in \tilde{H}\ot H$ is known to satisfy the following relation in $\tilde H \ot H^{\ot 2}$:
\begin{equation}
	X^{0,12}:=(\id \ot \Delta)(X)=(\R^{1,2})^{-1}X^{0,2}\R^{1,2}X^{0,1}\label{REqn}
\end{equation}
Here, $\tilde H$ has index 0 in the tensor product, and $\Delta$ denotes the coproduct of $H$.

There is a canonical action of the planar braid group $B_n(\mathbb{R}^2)$ on the $n$th tensor $V^{\ot n}$ power of any $H$-module $V$.  Given modules $M$ for $\tilde H$ and $V$ for $H$, equation \eqref{REqn} allows one to define a similarly canonical action of the punctured planar braid group $B_n(\mathbb{R}^2\backslash disc)$ on $M\ot V^{\ot n}$, and moreover to show that $\tilde H$ is universal for this action. We have:

\begin{thm}[\cite{Donin2003a}, Prop 10] Let $B$ be an algebra, and suppose that $X_B\in B\ot H$ satisfies relation \eqref{REqn}.  Then there is a unique homomorphism $\phi_B:\tilde H\rightarrow B$ such that $(\phi_B\ot\id)(X)=X_B$.\label{REuniv}
\end{thm}

The main goal of this paper is to define elliptic analogs of the reflection equation algebra.  The punctured elliptic braid group $B_n(T^2\backslash disc)$ is the free product of two copies of $B_n(\mathbb{R}^2\backslash disc)$, modulo certain relations.  In Section~\ref{sec:ell}  we construct an algebra $E_H$ as a certain crossed product of two copies of $\tilde H$, mimicking the cross relations of $B_n(T^2\backslash disc)$. We define canonical elements $X,Y\in E_H\ot H$ by
\[
X=\sum (e^i \ot 1) \ot e_i,\qquad
Y=\sum (1\ot e^i) \ot e_i,
\]
and characterize the cross relations on $E_H$ as follows:

\begin{thm}
	The cross relations of $E_H$ are equivalent to the following commutation relation in $E_H\ot H^{\ot 2}$ for $X,Y,\R$:
	\begin{equation}\label{eq:ellCom}		X^{0,1}\R^{2,1}Y^{0,2}=\R^{2,1}Y^{0,2}\R^{1,2}X^{0,1}\R^{2,1}
	\end{equation}
\end{thm}

We prove the following elliptic analog of Theorem \ref{REuniv}:

\begin{thm}\label{univ-prop}
Let $B$ be an algebra, and $X_B,Y_B \in B \ot H$ satisfying \eqref{REqn} individually, and \eqref{eq:ellCom} together. Then there exists a unique algebra morphism
\[
	\phi_B:E_H\longrightarrow B
\]
such that $X_B=(\phi_B \ot \id)(X)$ and $Y_B=(\phi_B \ot \id)(Y)$. Explicitly, $\phi_B$ is given by
\begin{align*}
\phi_B(x\ot 1)&= (\id \ot x)(X_B)&\phi_B(1\ot x)&=(\id \ot x)(Y_B).
\end{align*}
\end{thm}

Equation \eqref{eq:ellCom} can be used to define representations of $B_n(T^2\backslash disc)$ in the same way as \eqref{REqn} is used for $B_n(\mathbb{R}^2\backslash disc)$; see Theorem \ref{braid-actions}.  Recall that $B_n(T^2\backslash disc)$ carries a natural action of the punctured torus mapping class group, which is isomorphic to a certain central extension $\slz$ of $SL_2(\Z)$. In the case $H$ is a \emph{ribbon} Hopf algebra, we show that this extends to an action of $\slz$ on $E_H$.

When $H=U_q(\mathfrak{g})$, we produce degenerations of $E_H$ to the algebras of differential operators on $G$ and, upon further degeneration, on $\mathfrak{g}$.  Recall that the algebra of differential operators on an algebraic group $G$ can be constructed as a semi-direct product
\[
D(G)= U(\mf g)\ltimes O(G),
\]
where the action of $U(\mathfrak{g})$ on $O(G)$ is induced by that of $\mf g$ on $G$ by left invariant differential operators. This construction can be extended to any Hopf algebra and is known as the Heisenberg double~\cite{Semenov1994}. This is a semi-direct product $D_H=H\ltimes H^{\circ}$, where $H$ acts on its dual by the right coregular action.

In~\cite{Jordan2009}, canonical functors are constructed from the category of modules over the Heisenberg double of a quasi-triangular Hopf algebra to the category of modules over the (unpunctured) torus braid group. This relies upon an alternate construction -- due to Varagnolo-Vasserot \cite{Varagnolo2010} -- of the Heisenberg double of a quasi-triangular Hopf algebra, which uses the braided dual $\tilde H$ in place of $H^{\circ}$.  This presentation for the Heisenberg double also yields an isomorphism with the handle algebras introduced by Alekseev in~\cite{Alekseev1993} and studied further in~\cite{Alekseev1996,Alekseev1996a,Roche2002} (see Remark \ref{AGS-remark}).

Lifting the constructions of \cite{Jordan2009} to the unpunctured torus braid group, they can easily be re-interpreted as producing canonical elements $X$ and $Y$ in $D_H \ot H$, satisfying equations \eqref{REqn} and \eqref{eq:ellCom}.  Hence, Theorem \ref{univ-prop} yields a unique homomorphism $\Phi:E_H\to D_H$, compatible with the representations of the $B_n(T^2\backslash disc)$ on both sides.  The map $\Phi$ is an isomorphism if, and only if, $H$ is \emph{factorizable}. Since the quantum group $U_q(\mf g)$ is factorizable, we may identify the elliptic double $E_{U_q(\mathfrak{g})}$ with the algebra $D_q(G):=D_{U_q(\mathfrak{g})}$ of quantum differential operators on $G$.  

In particular we obtain an $\slz$ action on $D_q(G)$ by the above considerations.  One such automorphism of $D_q(G)$ we call the \emph{quantum Fourier transform}; its classical limit upon an appropriate degeneration is the classical Fourier transform on the Weyl algebra $D(\mf g)$.  We expect that our quantum Fourier transform for $D_q(G)$ will be compatible with that on the braided dual of $U_q(\mf g)$ defined in \cite{Lyubashenko1994}, realizing the braided dual as an $\slz$-equivariant $D_q(G)$-module. Studying this category of $\slz$-equivariant $D_q(G)$-modules more generally is an interesting direction of future research.

This paper is a companion to~\cite{Ben-Zvi2015}, in which we compute the value of a certain category valued 2-dimensional topological field theory attached to $H\modu$, and show that its value on a punctured torus is the category of $H$-equivariant $E_H$-modules.
\subsection*{Acknowledgments}
We are grateful to D. Ben-Zvi, and to all three authors of \cite{Calaque2009}, for their many helpful discussions and encouragement, and to P. Roche for bringing the article \cite{Alekseev1996} to our attention.

\section{The braided dual and its relatives}\label{sec:dual}
Let $(H,\R)$ be a quasi-triangular Hopf algebra, and denote by:
\begin{itemize}
\item $H^e=H^{coop}\ot H$ where $H^{coop}$ is $H$ with opposite comultiplication
\item $H^{[2]}$ the Hopf algebra which is $H\ot H$ as an algebra, and with coproduct given by
\[
\tilde\Delta(x\ot y)=(\R^{2,3})^{-1} (\tau^{2,3} \circ \Delta(x\ot y) ) \R^{2,3}
\]
\end{itemize}
where $\tau(a\ot b)=b\ot a$. Recall that the twist $H^F$ of $H$ by an invertible element $F\in H\ot H$ is the Hopf algebra with the same multiplication, and with coproduct given by 
\[
	\Delta^F(x)=F^{-1}\Delta(x)F.
\]
In order for $H^F$ to be co-associative, $F$ must satisfy two conditions:
\[
F^{12,3}F^{1,2}=F^{1,23}F^{2,3},\qquad (\epsilon \ot \id)(F)=(\id \ot \epsilon)(F)=1.
\]
Two twists $F,F'$ are \emph{equivalent} if there exists an invertible element $x\in H$, such that $\epsilon(x)=1$ and
\[
F'=\Delta(x)F(x^{-1}\ot x^{-1}).
\]
The following is standard (see \cite{Drinfeld1990}):
\begin{prop}
A twist induces a tensor equivalence $H\modu \rightarrow H^{F}\modu$. Equivalent twists leads to isomorphic tensor functors.
\end{prop}

It is easily checked that  $F=\R^{1,3}\R^{1,4} \in (H^e)^{\ot 2}$ is a twist, and that
\[
H^{[2],coop}=(H^e)^F.
\]
Let $D$ be the ``double braiding" $\R^{2,1}\R^{1,2}$.  Since $D\Delta(x)=\Delta(x)D$ for all $x$, we have: 
\[
H^{D}=H
\]
as Hopf algebras. Similarly, $H^{[2],coop}$ is in fact equal to $(H^e)^{F(D^{1,3})^k}$ for any $k\in \Z$, with $F$ as above.

Let $H^{\circ}$ be the restricted Hopf algebra dual of $H$. It has a natural $H$-bimodule structure, hence a $H^e$ left module structure given by:
\[
	(x \ot y)\lef f := f(S^{-1}(x)\,\cdot\, y)
\]
where $S$ is the antipode of $H$ and we use the fact that $S^{-1}$ is a Hopf algebra isomorphism $H^{coop}\rightarrow H_{op}$. It turns $H^{\circ}$ into an algebra in $H^e\modu$.
\begin{rmk}
Remember that the antipode of an Hopf algebra need not to be invertible in general, but this is implied by quasi-triangularity.
\end{rmk}
\begin{rmk}
We use the inverse of the antipode rather than the antipode itself because it is convenient to consider the canonical element as an invariant element of $H^\circ \ot H$, the image of $1\in\mathbb{C}$ under the evaluation map $\kk \rightarrow H^\circ \ot H$, which means that $H^\circ$ really denotes the \emph{left} dual of $H$ in the rigid monoidal category of $H$-modules.  This is slightly different from the convention used in~\cite{Donin2003a,Jordan2009} but it allows us to label tensor factors from left to right.
\end{rmk}
\begin{defi}
The $k$th twisted braided dual $\tilde H_k$ is the algebra image of $H^{\circ}$ via the tensor functor $H^e\modu\rightarrow H^{[2],coop}\modu$ given by the twist $F(D^{1,3})^k$. Explicitly, this is $H^{\circ}$ as a vector space, with multiplication given by
\[
x \cdot y = m ( \R^{1,3}\R^{1,4}(D^{1,3})^k\lef (x \ot y))
\]
where $m$ is the multiplication of $H^{\circ}$. This is an algebra in the category of $H^{[2],coop}$-module with the same action as above, namely
\[
	(x\ot y)\lef f=(u \mapsto f(S^{-1}(x)uy)).
\]
\end{defi}

\begin{rmk}
The algebra $\tilde H_0$ is usually called the reflection dual, the braided dual or the reflection equation algebra in the literature.
\end{rmk}
Let $X$ be the canonical element of $\tilde H_k\ot H$, that is the image of 1 under the coevaluation map $\kk\rightarrow \tilde H_k \ot H$. If $e_i$ is a basis of $H$ and $e^i$ the dual basis of $\tilde H_k \cong H^{\circ}$, then $X=\sum e^i \ot e_i$. If $H$ is infinite dimensional then $X$ lives in an appropriate completion of the tensor product.
\begin{prop}\label{prop:DKM1}
The element $X$ satisfies:
\begin{equation}\label{eq:reflec}
	X^{0,12}=D^k(\R^{1,2})^{-1}X^{0,2}\R^{1,2}X^{0,1}
\end{equation}
in $\tilde H_k \ot H^{\ot 2}$.
This implies that $X$ satisfies the reflection equation
\[
	\R^{2,1}X^{0,2}\R^{1,2}X^{0,1}=X^{0,1}\R^{2,1}X^{0,2}\R^{1,2}
\]
in $\tilde H_k \ot H^{\ot 2}$.
\end{prop}
The braided dual is in fact universal for this property in the following sense:
\begin{prop}\label{prop:DKM2}
	Let $B$ be an algebra and $X_B \in B\ot H$ satisfying equation~\eqref{eq:reflec} in $B\ot H^{\ot 2}$ for some $k\in \Z$. Then there exists a unique algebra morphism
	\[
\phi_B:\tilde H_k \longrightarrow B
	\]
	such that $(\phi_B \ot \id)(X)=X_B$. Explicitly, $\phi_B$ is given by
	\[
		H^{\circ}\cong \tilde H \ni f \longmapsto (f \ot id)(X).
	\]
\end{prop}
Propositions~\ref{prop:DKM1} and ~\ref{prop:DKM2} are proved in~\cite{Donin2003a} in the case $k=0$. The general proof is similar. Note that the fact that these axioms all leads to the same reflection equation, regardless of the value of $k$, essentially follows from the fact that the left hand side of~\eqref{eq:reflec} is invariant under conjugation by $D$.

Let $u=m((S\ot \id)(R^{2,1}))$ where $m$ is the multiplication of $H$. Then $\nu=uS(u)$ is central and satisfies
\[
\Delta(\nu)=D^{-2}(\nu \ot \nu)
\]
implying that
\[
D^{k-2}=\Delta(\nu)D^k (\nu^{-1}\ot \nu^{-1})
\]
meaning that $D^{k-2}$ and $D^k$ are equivalent. Therefore, they lead to isomorphic tensor functors, from which follows the following:
\begin{prop}
For any $k \in \Z$, the algebras $\tilde H_k$ and $\tilde H_{k+2}$ are isomorphic.
\end{prop}

Therefore, it is enough to consider $\tilde H_0$ and $\tilde H_1$. Moreover, if $H$ is a ribbon Hopf algebra, then by definition $\nu$ admits a central square root implying by a similar argument:
\begin{prop}
If $H$ is a ribbon Hopf algebra then all the $\tilde H_k$ are isomorphic.
\end{prop}
\begin{rmk}
For any $k$, equation~\eqref{eq:reflec} plays the same role in the reflection equation, as the hexagon axiom in the Yang-Baxter equation, encoding some kind of compatibility with the tensor product of $H$-modules.  Topologically, it corresponds to a ``strand doubling" operation for the additional generator of the braid group of the punctured plane. Formally, such an operation depends on the choice of a framing, while a ribbon element removes the dependence on the framing.
\end{rmk}
\section{The elliptic double}\label{sec:ell}
Let $T$ denote the following element in $(H^{[2],coop})^{\ot 2}$, which we identify as a vector space with $H^{\ot 4}$:
\[
	T=(\R^{3,2})^{-1}(\R^{3,1})^{-1} (\R^{4,2})^{-1}\R^{1,4}.
\]
\begin{prop}\label{prop:hexa}
	The element $T$ satisfies the hexagon axioms
\begin{align*}
(\id \ot \Delta_{H^{[2],coop}})T & = T^{1,3}T^{1,2} & (\Delta_{H^{[2],coop}}\ot \id)T=T^{1,3}T^{2,3}
\end{align*}
in $(H^{[2],coop})^{\ot 3}$.
\end{prop}
\begin{proof} This is a straightforward computation with the Yang-Baxter equation.  The computation is depicted in braids in Figure \ref{hex-reln}.
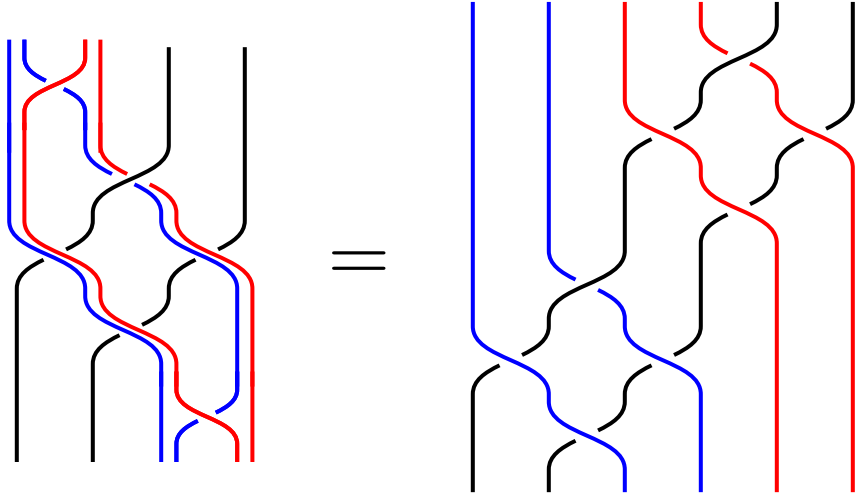
\begin{figure}[h]
\centering

\begin{tikzpicture}

\braid[number of strands=6,
 line width=1.5pt,
 style strands={1,2,5,6}{opacity=0},
] (braid2) at (0,-.1) s_1 s_2^{-1}  s_1-s_3  s_2 s_5 ;
\braid[number of strands=6,
 line width=1.5pt,
 style strands={1,2}{blue},
 style strands={3,4,5,6}{opacity=0}
] (braid2) at (-.1,-1.1) s_2^{-1} s_1-s_3 s_2 ;
\braid[number of strands=6,
 line width=1.5pt,
 style strands={3,4,5,6}{opacity=0},
 style strands={1,2}{red}
] (braid2) at (.1,-1.1) s_2^{-1} s_1-s_3 s_2 ;

\braid[number of strands=6,
 line width=1.5pt,
 style strands={1}{blue},
 style strands={2,3,4,5,6}{opacity=0}
] (braid2) at (-.1,0) s_5 ;
\braid[number of strands=6,
 line width=1.5pt,
 style strands={1,3,4,5,6}{opacity=0},
 style strands={2}{red}
] (braid2) at (.1,0) s_5 ;
\braid[scale=0.8, number of strands=6,
 line width=1.5pt,
 style strands={1}{blue},
 style strands={2,3,4,5,6}{opacity=0}
] (braid2) at (2.375,-5.5) s_5 ;
\braid[scale=0.8,number of strands=6,
 line width=1.5pt,
 style strands={1,3,4,5,6}{opacity=0},
 style strands={2}{red}
] (braid2) at (2.875,-5.5) s_5 ;

\braid[scale=0.8, line width=1.5pt, style strands={1}{red}, style strands={2}{blue}] (braid3) at (2.625,-5.5) s_1;
\braid[scale=0.8, line width=1.5pt, style strands={2}{red}, style strands={1}{blue}] (braid3) at (0.125,0) s_1^{-1};
\braid[scale=0.8, line width=1.5pt, style strands={1}{red}, style strands={2}{blue}] (braid3) at (2.625,-5.5) s_1;
\braid[scale=0.8, line width=1.5pt, style strands={2}{red}, style strands={1}{blue}] (braid3) at (0.125,0) s_1^{-1};

\node[scale=3] at (4.5,-3) {=};
\braid[
 line width=1.5pt,
 style strands={1,2}{blue},
 style strands={3,4}{red},
] (braid2) at (6,.5) s_4^{-1} s_3-s_5 s_4 s_2^{-1} s_1-s_3 s_2 ;
\end{tikzpicture}
\caption{A braid diagram proof of $(\id\ot\Delta)(T)=T_{1,3}T_{1,2}$.}\label{hex-reln}
\end{figure}

\end{proof}
\begin{cor}\label{EHk-cor}
The vector space $\tilde{H}_k^{\ot 2}$ carries an associative multiplication, in which $\tilde{H}_k \ot 1$ and $1\ot \tilde{H}_k$ are sub-algebras, and the cross relations are given by
\[
(1\ot g)(f\ot 1) = T \rhd (f\ot g).
\]
\end{cor}
While this is well-known, we include a proof here for the reader's convenience.
\begin{proof}
It suffices to check associativity on pure tensors in $\tilde{H}_k$, as these span all of $E_H^{(k)}$.  Since $\tilde{H}_k$ is an associative algebra, the only types of expressions on which it remains to check associativity are of the form $(1\ot g) \ot (1 \ot h) \ot (f\ot 1)$ and $(1\ot g)\ot (h \ot 1) \ot (f \ot 1)$.  Write $T=\sum t_i\ot t_i'$.  For the first case, we have:
	\begin{multline*}
	(m \circ (m\ot \id))((1\ot g) \ot (1 \ot h) \ot (f\ot 1)) \\=\sum_i\sum_j ((t_it_j)\rhd f) \ot (t'_i\rhd g)(t'_j \rhd h) = \sum ((t_i\rhd f)\ot\Delta(t_i)\rhd gh),
	\end{multline*}
so that associativity follows from the second equation in Proposition~\ref{prop:hexa}.  The second case follows similarly.
\end{proof}

\begin{defi} 
We denote by $E_H^{(k)}$ the algebra given by Corollary~\ref{EHk-cor}.
\end{defi}
Choose a basis $(e_i)_{i\in I}$ of $H$ and define $X,Y\in E_H^{(k)}\ot H$ by
\[
X=\sum e^i \ot 1 \ot e_i,\qquad Y=\sum 1\ot e^i \ot e_i,
\]
where we use the vector space identification $E_H^{(k)} \cong \tilde H^{\ot 2}$. The main result of this section is the following:
\begin{thm}\label{EHrelns}
	The cross relations of $E_H$ are equivalent to the commutation relation in $E_H\ot H^{\ot 2}$ for $X,Y,\R$:
	\begin{equation*}
		X^{0,1}\R^{2,1}Y^{0,2}=\R^{2,1}Y^{0,2}\R^{1,2}X^{0,1}\R^{2,1}.
	\end{equation*}
\end{thm}
\begin{proof}
By definition every element $f \in \tilde H_k$ can be written as
\[
	f=\sum e^if(e_i)
\]
hence the product $gf$ in $E_H^{(k)}$ is obtained by applying $(\id_{E_H^{(k)}}\ot f \ot g)$ to
\[
	Y^{0,2}X^{0,1}
\]
and $fg$ by applying the same element to 
\[
X^{0,1}Y^{0,2}.
\]
Therefore all commutations relation can be gathered into a ``matrix" equation
\begin{equation}\label{eq:matEll}
	Y^{0,2}X^{0,1}= T\lef_0 X^{0,1}Y^{0,2}
\end{equation}
where $T$ acts on the $E_H^{(k)}$ (i.e. 0th) component. We recall the following identities:
\begin{align}\label{eq:R1}
	\R^{-1}&=(S\ot \id)(\R)=(\id \ot S^{-1})(\R)  .
\end{align}
Applying $S^{-1}$ to the first factor of the relation $(S \ot \id)(R)R=1$, setting $\R=\sum r_1 \ot r_2 =\sum r_1'\otimes r_2'$ -- using apostrophes to distinguish between copies of $\R$ --  one has the following useful identity (note the order of the terms):
\begin{equation}\label{eq:R2}
\sum S^{-1}(r_1)r_1' \ot r_2'r_2=1.
\end{equation}
Then equation~\eqref{eq:matEll} reads, in coordinates:
\begin{multline}\label{eq:temp}
\left((1\ot e^j)(e^i \ot 1)\right) \ot e_i \ot e_j\\=  \left( (r_2r_1' \ot r_2'''' r_2'' \ot S(r_1'''')S(r_1) \ot S(r_1'')r_2')\lef e^i \ot e^j \right ) \ot e_i \ot e_j.
\end{multline}

The left $H^{[2]}$ action on $\tilde H_k$ is by definition dual to the right $H^{[2]}$ action on $H$, therefore:
\[
	\sum ((x \ot y)\lef e^i) \ot e_i= \sum e^i \ot S^{-1}(x)e_iy
\]
Using this, equation~\eqref{eq:temp} can be rewritten
\[
\left((1\ot e^j)(e^i \ot 1)\right) \ot e_i \ot e_j=  e^i \ot e^j  \ot S^{-1}(r_1')S^{-1}(r_2)e_ir_2''''r_2'' \ot r_1r_1''''e_jS(r_1'')r_2'.
\]
Then, using the $R$-matrix relations~\eqref{eq:R1} and \eqref{eq:R2} to move elements from the right hand side to the left hand side (and reassigning apostrophes for the sake of clarity) we obtain:
\[
\left((1\ot e^j)(e^i \ot 1)\right) \ot r_2r_1'e_ir_2'' \ot r_1e_jr_2'r_1''=  e^i \ot e^j  \ot e_ir_2 \ot r_1e_j
\]
which is exactly~\eqref{eq:ellCom}.
\end{proof}

\begin{rmk}\label{AGS-remark} 
	If $H$ is semi-simple, then as a vector space $\tilde H_k\cong H^{\circ}$ has a Peter-Weyl decomposition
	\[
		\tilde H_k=\bigoplus V^* \ot V
	\]
where the sum is over representatives of finite dimensional simple $H$-modules. Under this identification, the relations of Theorem \ref{EHrelns} coincide with those of the graph algebra of the punctured torus of~\cite[Def. 12]{Alekseev1993}.\end{rmk}

Equation~\eqref{eq:ellCom} is a defining relation for $E_H^{(k)}$, in the following sense:
\begin{cor}\label{cor:universal}
Let $B$ be an algebra, and $X_B,Y_B \in B \ot H$ satisfying both the axiom~\eqref{eq:reflec} and equation~\eqref{eq:ellCom} (with $X$ and $Y$ replaced by $X_B$ and $Y_B$). Then there exists a unique algebra morphism
\[
	\phi_B:E_H^{(k)} \longrightarrow B
\]
such that $X_B=(\phi_B \ot \id)(X)$ and $Y_B=(\phi_B \ot \id)(Y)$. Explicitly, $\phi_B$ is given by
\begin{align*}
\phi_B(x\ot 1)&= (\id \ot x)(X_B)&\phi_B(1\ot x)&=(\id \ot x)(Y_B).
\end{align*}
\end{cor}
\section{Braid group and mapping class group actions}
In this section we construct representations of the punctured torus braid group from $E^{(k)}_H$. 
First, we have:
\begin{defi}
The punctured elliptic braid group $B_n(T^2\backslash disc)$ is the fundamental group of the configuration space of $n$ points in $T^2\backslash disc$.
\end{defi}
\begin{prop}
The group $B_n(T^2\backslash disc)$ is generated by $X_1,\dots, X_n,Y_1,\dots,Y_n,\sigma_1,\dots,\sigma_{n-1}$ with relations:
\begin{itemize}
\item the $X_i$'s (resp. $Y_i'$s) pairwise commute,
\item the planar braid relation for the $\sigma_i$'s,
\item the following cross relations:
\begin{align}
X_{i+1}&=\sigma_i X_i \sigma_i& Y_{i+1}=\sigma_i Y_i \sigma_i \\
X_1 Y_2 &= Y_2 X_1 \sigma_1^2
\end{align}
\end{itemize}
\end{prop}

The results of the previous section easily imply:
\begin{thm}\label{braid-actions}
There exists a unique group morphism
\[
\phi: B_n(T^2\backslash disc) \longrightarrow (E_H^{(k)} \ot H^{\ot n})^{\times}\rtimes S_n
\]
given by
$$X_1 \longmapsto X^{0,1},\qquad Y_1 \longmapsto Y^{0,1},\qquad \sigma_i \longmapsto (i,i+1)\R^{i,i+1}.$$
\end{thm}
\begin{proof}
The first two set of cross relations can obviously be taken as a definition of $X_i,Y_i$ for $i>1$. That these operators pairwise commute follows from the reflection equation and the Yang-Baxter equation. The last cross relation is nothing but the defining equation~\eqref{eq:ellCom} of $E_H^{(k)}$.
\end{proof}

Let $\slz$ denote the group generated by $A,B,Z$ with relations:
\begin{align}
A^4=(AB)^3&=Z, & (A^2,B)&=1.
\end{align}
Clearly, $Z$ is central, so this is a central extension,
\[
1 \rightarrow \Z \rightarrow \slz \rightarrow SL_2(\Z)\rightarrow 1.
\]
\begin{prop}
	The group $\slz$ acts on $B_n(T^2\backslash disc)$ in the following way:
	\begin{align*}
		A\cdot \sigma_i &= \sigma_i & B \cdot \sigma_i &= \sigma_i\\
		A\cdot X_1&=Y_1 & A\cdot Y_1 &= Y_1X_1^{-1}Y_1^{-1}\\
		B \cdot X_1&=X_1 & B \cdot Y_1&=Y_1 X_1^{-1}.
	\end{align*}
\end{prop}

\begin{prop}\label{prop:refsl2}
Let $B$ be an algebra and $(X_B,Y_B) \in B \ot H$ satisfying equation~\eqref{eq:ellCom} and axioms~\eqref{eq:reflec} with $k=1$. Then, so does $(X_B,Y_BX_B^{-1})$ and $(Y_B,Y_BX_B^{-1}Y_B^{-1})$.
\end{prop}
\begin{proof}
	Equation~\eqref{eq:ellCom} is exactly one of the defining relation of $B_{1,n}^1$ so that it is satisfied follows from the previous proposition. So we just have to check that $Y_BX_B^{-1}$ and $Y_BX_B^{-1}Y_B^{-1}$ satisfies~\eqref{eq:reflec} with $k=1$. This is a direct computation:
	\begin{align*}
		(Y_BX_B^{-1})^{0,12}&= \R^{2,1}Y_B^{0,2}\R^{1,2}Y_B^{0,1} (X_B^{0,1})^{-1} (\R^{1,2})^{-1} (X_B^{0,2})^{-1} (\R^{2,1})^{-1}\\
&= \R^{2,1}Y_B^{0,2}\R^{1,2}Y_B^{0,1} (\R^{1,2})^{-1}(X_B^{0,2})^{-1} (\R^{2,1})^{-1}(X_B^{0,1})^{-1}\R^{2,1}(\R^{2,1})^{-1}\\
&= \R^{2,1}Y_B^{0,2}\R^{1,2}(\R^{1,2})^{-1}(X_B^{0,2})^{-1}\R^{1,2}Y_B^{0,1}\R^{2,1}(\R^{2,1})^{-1}(X_B^{0,1})^{-1}\\
&= \R^{2,1}Y_B^{0,2}(X_B^{0,2})^{-1}\R^{1,2}Y_B^{0,1}(X_B^{0,1})^{-1},
	\end{align*}
where at lines 2 and 3 we use the reflection equation and the elliptic commutation relation respectively. The second part is proved by doing the exact same computation replacing $Y_B$ by $Y_BX_B^{-1}$ and $X_B$ by $Y_B$.
\end{proof}
\begin{cor}\label{cor:action}
There is an action of $\slz$ on $E_H^{(1)}$, uniquely determined by its action on canonical elements $X,Y$ as follows:
\begin{align*}
		A\cdot X&=Y, & A\cdot Y &= YX^{-1}Y^{-1},\\
		B \cdot X&=X, & B \cdot Y&=Y X^{-1}.
	\end{align*}
Moreover, the action is compatible with the $\slz$-action on $B_n(T^2\backslash disc)$, 
\end{cor}
\begin{proof}
This follows from Proposition~\ref{prop:refsl2} together with the universal property stated in Corollary~\ref{cor:universal}.
\end{proof}
\section{Relation with the Heisenberg double and quantum Fourier transform}

Since $\tilde H_0$ is a $H^{[2],coop}$-module algebra, one can form the semi-direct product $\tilde H \rtimes H^{[2],coop}$. It is easily checked that $H \ot 1 \subset H^{[2],coop}$ is a coideal subalgebra, hence the following definition makes sense:
\begin{defi}
	The Heisenberg double $D_H$ is the subalgebra $\tilde H_0\rtimes (H\ot 1)$.
\end{defi}
\begin{rmk}
The standard definition of the Heisenberg double involves $H^e$ and the usual dual, instead of $H^{[2]}$ and the braided dual. However, it is shown in \cite{Varagnolo2010} that these two algebras are isomorphic.
\end{rmk}

Clearly, the double braiding $\R^{2,1}\R^{1,2}$ satisfies axiom~\eqref{eq:reflec} with $k=0$. This is a manifestation of the embedding of the cylinder braid group on $n$ strands into the ordinary braid group on $n+1$ strands.  Let $\phi_H$ be the factorization map
\begin{align*}\phi_H: \tilde{H}_0 &\to H,\\
f &\mapsto (f \ot id)(\R^{2,1}\R^{1,2}).\end{align*}.
We have: 
\begin{thm}\cite{Jordan2009}
	The canonical element $X \in D_H \ot H$ together with the image of the double braiding under the inclusion $H\ot H \rightarrow D_H \ot H$ satisfy the commutation relation~\eqref{eq:ellCom}.
\end{thm}
\begin{cor}
There exists a canonical algebra map from the elliptic double to the Heisenberg double, given by the identity on the first $\tilde H_0$ component and defined on the second component by the factorization map $\phi_H$.
\end{cor}
\begin{proof}
	It follows from the universal property of Corollary~\ref{cor:universal}.
\end{proof}
\begin{defi}
A quasi-triangular Hopf algebra is called factorizable if $\phi_H$ is injective.
\end{defi}

Let $I_H$ be the image of $\phi_H$ and let $D_H'$ be the subalgebra $\tilde H \rtimes (I_H \ot 1)$ of $D_H$. 
\begin{thm}
If $H$ is a factorizable Hopf algebra, then $D_H'$ is isomorphic as an algebra to $E_H^{(0)}$.
\end{thm}
\begin{proof}
	The algebra map $E_H^{(0)}\rightarrow D_H$ is given by $\id \ot \phi_H$. Since $H$ is factorizable this map is injective, and its image is $D_H'$ by definition.
\end{proof}
Let $G$ be a reductive algebraic group, $\mf g$ its Lie algebra and $U=U_q(\mf g)$ the corresponding quantum group. Recall (see e.g.~\cite[Chap.~9]{Chari1994}) that this is a quasi-triangular Hopf algebra\footnote{This is not quite true since the R-matrix does not belong to $U_q(\mf g)^{\ot 2}$ but only to a certain completion of it, but it is still enough for our purpose} over $\C(q)$ for $q$ a variable which deform the enveloping algebra of $\mf g$. Denote by $U'=U_q(\mf g)'$ its ad-locally finite part.
\begin{thm}[\cite{Baumann1998,Reshetikhin1988}]
$U$ is a factorizable ribbon Hopf algebra, and the image of the factorization map $(U^*) \rightarrow U$ is $U'$.
\end{thm}
Let $D_q(G)$ be the subalgebra $\tilde U\rtimes U'$ of the Heisenberg double of $U$. It is a deformation of the algebra of differential operators on $G$. Thanks to the above theorem, $D_q(G)$ is isomorphic to $E_{U}^{(0)}$ which is itself isomorphic to $E_U^{(1)}$. Altogether this implies the following: 
\begin{cor}
	The isomorphism $D_q(G)\cong E_{U}^{(1)}$ together with the formulas of Corollary~\ref{cor:action} yield an action of $\slz$ on $D_q(G)$ by algebra automorphism.
\end{cor}

\section{Relation to classical Fourier transform}
In this section we show how the Weyl algebra of $\mf g$ and the classical Fourier transform can be obtained both directly as the elliptic double of a certain Hopf algebra and via an appropriate degeneration of the elliptic double of the corresponding quantum group. Let $U_{\h}(\mf g)$ be the ``formal'' version of the quantum group. This a topological quasi-triangular Hopf algebra over $\C[[\h]]$, where $\h$ is a formal variable, deforming the enveloping algebra of $\mf g$ and whose definition can be found, e.g., in~\cite[Chap.~6]{Chari1994}. Since directly taking the classical (i.e. $\hbar=0$) limit of the elliptic commutation relation gives the commutative algebra $S(\mf g)^{\ot 2}$ we will have to consider a slightly more complicated degeneration.

Let $S(\mf g)$ denote the symmetric algebra on $\mf g$, equipped with its standard co-product $\Delta(X) = X\ot 1 + 1 \ot X$ for $X\in \mf g$, making it a commutative, cocommutative Hopf algebra. Let $r\in \mf g^{\ot 2}$ denote the quasi-classical limit of the R-matrix of $U_{\h}(\mf g)$, i.e.:
\[
\R=1+\h r+O(\h^2).
\]
Then, in a straightforward way, the completion of the symmetric algebra $(\widehat{S}(\mf g),\R_0=\exp(r))$ is a quasi-triangular, factorizable Hopf algebra\footnote{Here the tensor product is the topological one, i.e. $\widehat{S}(\mf g)^{\ot 2}:=\widehat{S}(\mf g\times\mf g)$}. Let $t=r+r^{2,1} \in S^2(\mf g)^{\mf g}$ and let $C$ denote the corresponding Casimir element, i.e. $C=m(t)$ where $m$ is the multiplication of $S(\mf g)$. Then $\nu_0=\exp(-C/2)$ is a ribbon element. Since $\R_0 \not\in S(\mf g)^{\ot 2}$, $S(\mf g)$ is not strictly speaking a ribbon Hopf algebra, but the construction of the elliptic double is still well defined in this situation. 

Let $D(\mf g)$ be the algebra of differential operators on $\mf g$, i.e. the Weyl algebra. As a vector space it is $S(\mf g^*)^{\ot 2}$, the two copies of $S(\mf g^*)$ are subalgebras and the cross relations are:
\begin{equation}
\forall f,g \in \mf g^*, [f \ot 1,1\ot g]=\langle f,g\rangle \label{WeylDefRel}
\end{equation}
where $\langle\ ,\ \rangle$ is the pairing on $\mf g^*$ induced by $t$. The first result of this section is:
\begin{prop}
The 0th elliptic double of $(S(\mf g),\R_0)$ is isomorphic to the Weyl algebra $D(\mf g)$ and the action of the generator $A$ of $\slz$ coincides with the classical Fourier transform. That is, on generators $(f,g)\in \mf g^*\times \mf g^*\subset D(\mf g)$, we have, 
\[
A(f,g) = (-g,f).
\]
The operator $B$ acts by
\[
	B(f,g)=(f-g,g).
\]
\end{prop}

\begin{proof}
	Let $E$ be the $0$th elliptic double of $(S(\mf g),\R_0)$. Let $e_i$ be a basis of $\mf g$, $e^i$ the dual basis of $\mf g^*$ and define $x,y \in E \ot U(\mf g)$ by
	\begin{align*}
		x&=\sum e^i \ot 1 \ot e_i & y&=\sum 1\ot e^i \ot e_i.
	\end{align*}
	The restricted dual of $S(\mf g)$ is $S(\mf g^*)$ and the images of the corresponding canonical elements in $E\ot S(\mf g)$ are $X=\exp(x)$ and $Y=\exp(y)$ respectively. Since $S(\mf g)$ is commutative, equation~\eqref{eq:reflec} reduces to the standard relation,
\[
(\id \ot \Delta)(X)=X^{0,1}X^{0,2}
\]
in $(S(\mf g^*)\ot 1) \ot S(\mf g)^{\ot 2}\subset E\ot S(\mf g)^{\ot 2}$, hence the braided dual and the restricted dual coincide. Likewise, the defining equation of the elliptic double reduces to:
\[
	(X^{0,1},Y^{0,2})=\R_0^{2,1}\R_0^{1,2}
\]
in $E\ot S(\mf g)^{\ot 2}$, where $(a,b)=aba^{-1}b^{-1}$. Since
\[
[x^{0,1},t^{1,2}]=[y^{0,2},t^{1,2}]=0,
\]
this equation is equivalent to:
\[
[x^{0,1},y^{0,2}]=t^{1,2}.
\]
Applying $f$ and $g$ to the first and second components, respectively, of the above equation gives the defining relations \eqref{WeylDefRel} of $D(\mf g)$.

Since $(S(\mf g),\R_0)$ is ribbon, $E_{S(\mf g)}^{(0)}$ is isomorphic to $E_{S(\mf g)}^{(1)}$. Pulling back the action of the $A$ generator of $\slz$ through this isomorphism, we find:
\begin{align*}
x &\mapsto y & y &\mapsto Y^{-1}(-x+(1\ot C))Y.
\end{align*}
It is easily seen that the cross relations of $D(\mf g)$ implies
\[
Y^{-1}xY=x+(1\ot C).
\]
Hence $A$ maps $x$ to $y$ and $y$ to $-x$. 

Pulling back the $B$ action through this isomorphism one get
\begin{align*}
	x &\mapsto x & y&\mapsto \log(e^ye^{-x}e^{1\ot C/2}).
\end{align*}
Since
\[
	[x,y]=1\ot C
\]
and since $1\ot C$ commutes with $x$ and $y$, the Baker--Campbell--Hausdorff formula implies that
\[
\log(e^ye^{-x}e^{1\ot C/2})=y-x
\]
as required.
\end{proof}

\begin{rmk} Since $A^4$ acts as the identity, the above action of $\widetilde{SL_2(\mathbb{Z})}$ on $D(\mf g)$ factors through an action of $SL_2(\mathbb{Z})$. It coincides with the one coming from an homomorphism $SL_2(\mathbb{Z})\to Sp(\mf g \oplus \mf g)$, the latter being the group of linear symplectomorphisms of the vector space $\mf g \oplus \mf g$, equipped with the symplectic form coming from the Killing form.\end{rmk}

	\begin{rmk} It is interesting to ask whether the action of $\widetilde{SL_2(\mathbb{Z})}$ on $D_q(G)$ can be degenerated to an action on $D(G)$, not just to $D(\mf g)$.  The degeneration procedure for obtaining $D(G)$ from $D_q(G)$ is not compatible, however, with the $\widetilde{SL_2(\mathbb{Z})}$-action; hence, a naive attempt at re-creating the procedure for $D(\mf g)$ will not work.  This is not surprising, as there is not a good notion of Fourier transform on $D(G)$, essentially because the cotangent bundle $T^*G = G\times \mf g^*$ has fewer symplectomorphisms than $T^*\mf g = \mf g\times \mf g^*\cong \mf g\oplus \mf g$.
	\end{rmk}

Let $U_{\h^2}(\mf g)$ be the $\C[[\h]]$-Hopf algebra obtained by formally replacing $\h$ by $\h^2$ in the definition of the product, the coproduct and the R-matrix of $U_{\h}(\mf g)$. Denote by $\delta_n$ the map $(\id-\epsilon)^{\ot n}\circ \Delta^n$ where $\epsilon$ is the counit of $U_{\h^2}(\mf g)$. Denote by $\widehat{U}$ the quantum formal series Hopf algebra (QFSHA) attached to $U_{\h^2}(\mf g)$, i.e. the sub-algebra
\[
	\widehat{U}=\left\{x \in U_{\h^2}(\mf g),\ \delta_n(x) \in \h^nU_{\h^2}(\mf g),\ \forall n\geq 0\right\}
\]
It is known~\cite{Drinfeld1987,Gavarini2002} that $\widehat{U}$ is a flat deformation of $\widehat{S}(\mf g)$. Hence, choose a $\C[[\h]]$-module identification
\[
\psi:\widehat{U}\longrightarrow \widehat{S}(\mf g)[[\h]]
\]
which is the identity modulo $\h$, and let $U\subset \widehat{U}$ be the preimage under $\psi$ of $S(\mf g)[[\h]]$.
\begin{prop}\label{prop:modh}
We have the following:
\begin{enumerate}
\item $U$ is a Hopf algebra.
\item We have canonical bialgebra isomorphisms:
\begin{align*}
\widehat U /(\h) &\cong \widehat{S}(\mf g), & U/(\h) & \cong S(\mf g).
\end{align*}
\item The R-matrix of $U_{\h^2}(\mf g)$ belongs to $\widehat{U}^{\ot 2}$ and its image in $\widehat{S}(\mf g)^{\ot 2}$ is $\R_0$.
\end{enumerate}
\end{prop}
One can therefore consider the 0th elliptic double of $U$. A direct consequence of the above proposition is then:
\begin{cor}
	The algebra $E_U$ is a flat deformation of the Weyl algebra $D(\mf g)$, and the $\slz$-action on $E_U$ degenerates to the $\slz$-action on $D(\mf g)$. In particular, the quantum Fourier transform degenerates to the classical one. 
\end{cor}
\begin{proof}[Proof of Prop.~\ref{prop:modh}]
	All of this can be checked explicitly. A more conceptual argument is as follows: recall that $(\mf g, \mu,\delta,r)$ is a quasi-triangular Lie bialgebra, where we denote by $\mu$ its bracket and by $\delta$ its co-bracket. The quantum group $U_{\h^2}(\mf g)$ is obtained by applying an Etingof--Kazhdan quantization functor~\cite{Etingof1996} to the $\C[[\h]]$-quasi-triangular Lie bialgebra $(\mf g[[\h]], \mu,\h^2 \delta,\h^2 r)$. On the other hand, $\widehat{U}$ is the quasi-triangular Hopf algebra obtained by applying the same functor to the quasi-triangular Lie bialgebra $(\mf g[[\h]], \h\mu,\h\delta,r)$.  The QFSHA construction is the lift of the inclusion,
	\[
(\mf g[[\h]], \h\mu,\h\delta,r) \longrightarrow (\mf g[[\h]], \mu,\h^2\delta,\h^2r),
	\]
	given by $x\mapsto \h x$ (since $r\in \mf g^{\ot 2}$, its image is indeed $\h^2r$).

One can show that the product, the coproduct and the antipode on $\widehat{U}$ restrict to a well-defined Hopf algebra structure on $U$.
By construction, the reduction modulo $\h$ of $\widehat{U}$ is the quantization of the $\C$-quasi-triangular Lie bialgebra,
\[
(\mf g[[\h]], \h\mu,\h\delta,r)/(\h)\cong (\mf g,0,0,r),
\]
which is easily seen to be $(\widehat{S}(\mf g),\R_0)$.
\end{proof}


\end{document}